\documentclass[10pt]{amsart}
\usepackage[leqno]{amsmath}
\usepackage{amssymb,latexsym,soul,cite,amsthm,color,enumitem,graphicx,mathtools,microtype,accents}
\usepackage[colorlinks=true,urlcolor=cobalt,citecolor=cobalt,linkcolor=cobalt,linktocpage,pdfpagelabels,bookmarksnumbered,bookmarksopen]{hyperref}
\definecolor{cobalt}{rgb}{0.0, 0.28, 0.67}
\usepackage[english]{babel}
\usepackage[left=2.5cm,right=2.5cm,top=2.5cm,bottom=2.5cm]{geometry}

\numberwithin{equation}{section}

\newtheorem{theorem}{Theorem}[section]
\theoremstyle{plain}
\newtheorem{lemma}[theorem]{Lemma}
\theoremstyle{plain}
\newtheorem{proposition}[theorem]{Proposition}
\theoremstyle{plain}

\newtheorem{definition}[theorem]{Definition}
\theoremstyle{definition}
\newtheorem{remark}[theorem]{Remark}
\newtheorem{example}[theorem]{Example}

\newcommand{\N}{{\mathbb N}}

\newcommand{\R}{{\mathbb R}}
\newcommand{\eps}{\varepsilon}
\newcommand{\beq}{\begin{equation}}
\newcommand{\eeq}{\end{equation}}
\renewcommand{\le}{\leqslant}
\renewcommand{\ge}{\geqslant}

\newcommand{\w}{W^{s,p}_0(\Omega)}
\newcommand{\fpl}{(-\Delta)_p^s\,}
\newcommand{\ds}{{\rm d}_\Omega^s}
\newcommand{\cs}{C_s^0(\overline\Omega)}

\makeatletter
\newcommand{\lenomode}{\tagsleft@true}
\newcommand{\reqnomode}{\tagsleft@false}
\makeatother

\newenvironment{enumroman}{\begin{enumerate}

}{\end{enumerate}}

\title[Fractional $p$-Laplacian with parametric reaction]{Bifurcation-type results for the fractional $p$-Laplacian\\ with parametric nonlinear reaction}

\author[S.\ Frassu, A.\ Iannizzotto]{Silvia Frassu, Antonio Iannizzotto}

\address[S.\ Frassu, A.\ Iannizzotto]{Department of Mathematics and Computer Science
\newline\indent
University of Cagliari
\newline\indent
Via Ospedale 72, 09124 Cagliari, Italy}
\email{silvia.frassu@unica.it, antonio.iannizzotto@unica.it}


\begin{document}

\begin{abstract}
We study a Dirichlet problem driven by the degenerate fractional $p$-Laplacian and involving a nonlinear reaction, which depends on a positive parameter. The reaction is assumed to be $(p-1)$-sublinear near the origin and $(p-1)$-superlinear at infinity (including the concave-convex case). Following a variational approach based on a combination of critical point theory and suitable truncation techniques, we prove a bifurcation-type result for the existence of positive solutions.
\end{abstract}

\maketitle

\begin{center}
Version of \today\
\end{center}

\section{Introduction and main result}\label{sec1}

\noindent In this paper, we deal with the following Dirichlet problem for a nonlinear, nonlocal equation:
\[(P_\lambda) \ \ \ \ \ \ \ \begin{cases}
\fpl u = f(x,u,\lambda)& \text{in $\Omega$} \\
u>0& \text{in $\Omega$} \\
u=0 & \text{in $\Omega^c$.}
\end{cases}\]
Here $\Omega\subset\R^N$ ($N\ge 2$) is a bounded domain with $C^{1,1}$ boundary, $p \ge 2$, $s\in(0,1)$ s.t.\ $N>ps$, and the leading operator is the degenerate fractional $p$-Laplacian, defined for all $u:\R^N\to\R$ smooth enough and all $x\in\R^N$ by
\[\fpl u(x)=2\lim_{\eps\to 0^+}\int_{B_\eps^c(x)}\frac{|u(x)-u(y)|^{p-2}(u(x)-u(y))}{|x-y|^{N+ps}}\,dy\]
(which in the linear case $p=2$ reduces to the fractional Laplacian, up to a dimensional constant). The reaction $f:\Omega\times\R\times\R_0^+\to\R$ is a Carath\'eodory mapping, subject to a subcritical growth condition on the real variable, and depending in a general manner on a parameter $\lambda>0$. Our hypotheses on the reaction include a $(p-1)$-sublinear behavior near the origin and a $(p-1)$-superlinear one at infinity, along with a quasi-monotonicity condition and several conditions on the $\lambda$-dependance.
\vskip2pt
\noindent
Under such assumptions, we prove a bifurcation-type result for problem $(P_\lambda)$, namely, our problem admits at least two positive solutions for $\lambda$ below a certain threshold $\lambda^*>0$, at least one solution for $\lambda=\lambda^*$, and no solution for $\lambda>\lambda^*$. In addition, we study the behavior of solutions as $\lambda\to\lambda^*$.
\vskip2pt
\noindent
Our reaction embraces the model case of the concave-convex reaction introduced in \cite{ABC}, i.e., the following pure power map with exponents $1<q<p<r$:
\[t\mapsto \lambda t^{q-1}+t^{r-1} \ (t>0).\]
Nonlocal elliptic equations driven by the fractional $p$-Laplacian with concave-convex reactions are investigated, for instance, in \cite{BM,CM,DHS,K,LL1}. Other existence and bifurcation results for problems with several parametric reactions can be found in \cite{IMP,PSY,XZR,ZYY}. These are indeed only a few recent references out of a vast and increasing literature on fractional $p$-Laplacian equations, motivated by both intrinsic mathematicand interest and applications in game theory and nonlinear Dirichlet-to-Neumann operators, see \cite{BCF,W}.
\vskip2pt
\noindent
Here we try to keep the $\lambda$-dependence as general as possible, assuming at the same time several conditions on the behavior of $f(\cdot,\cdot,\lambda)$. The main novelty of the present work, in  the framework of nonlocal equations, is that we consider general parametric reactions rather than focusing on pure power type maps. Also, with respect to previous results, we gain new monotonicity and convergence properties of the solutions with respect to $\lambda$.
\vskip2pt
\noindent
Our approach is entirely variational, based on critical point theory and suitable truncation techniques, and follows mainly \cite{IP}. In particular, we shall often use two recent results on equivalence between Sobolev and H\"older minima of the energy functional from \cite{IMS1}, and on strong maximum and comparison principles from \cite{IMP}. This will allow us to establish a general sub-supersolution principle for problem $(P_\lambda)$ and to slightly relax the assumptions on the mapping $\lambda\mapsto f(x,t,\lambda)$ with respect to \cite{IP}. Also, in the proof of the nonexistence result we will employ a recent anti-maximum principle proved in \cite{FI1}.
\vskip2pt
\noindent
Our precise hypotheses on the reaction $f$ are the following:
\begin{itemize}[leftmargin=1cm]
\item[${\bf H}$] $f:\Omega\times\R\times\R^+_0\to\R$ is a Carath\'eodory map s.t.\ $f(x,0,\lambda)=0$ for a.e.\ $x\in\Omega$ and all $\lambda>0$, and for all $(x,t,\lambda)\in\Omega\times\R\times\R^+_0$ we set
\[F(x,t,\lambda) = \int_0^t f(x,\tau,\lambda)\,d\tau.\]
Also, the following conditions hold:
\begin{enumroman}
\item\label{h1} there exist $c_1>0$, $r\in(p,p^*_s)$, and for all $\lambda>0$ a function $a_\lambda\in L^\infty(\Omega)_+$ s.t.\ $\lambda\mapsto\|a_\lambda\|_\infty$ is locally bounded, $\|a_\lambda\|_\infty\to 0$ as $\lambda\to 0$, and for a.e.\ $x\in\Omega$ and all $t\ge 0$, $\lambda>0$
\[|f(x,t,\lambda)| \le a_\lambda(x)+c_1t^{r-1};\]
\item\label{h2} for all $\lambda>0$, uniformly for a.e.\ $x\in\Omega$
\[\lim_{t\to\infty}\frac{f(x,t,\lambda)}{t^{p-1}} = \infty;\]
\item\label{h3} there exist $\rho\in (\frac{N}{ps}(r-p),\,p^*_s)$, and for all $\Lambda>0$ a number $\theta>0$ s.t.\ for all $\lambda\in(0,\Lambda]$, uniformly for a.e.\ $x\in\Omega$
\[\liminf_{t\to\infty}\frac{f(x,t,\lambda)t-pF(x,t,\lambda)}{t^\rho} > \theta;\]
\item\label{h4} for all $\Lambda>0$ there exist $c_2,\delta>0$, $q\in(1,p)$ s.t.\ for a.e.\ $x\in\Omega$ and all $t\in[0,\delta]$, $\lambda\ge\Lambda$
\[f(x,t,\lambda) \ge c_2t^{q-1};\]
\item\label{h5} for all $T,\Lambda>0$ there exists $\sigma>0$ s.t.\ for a.e.\ $x\in\Omega$ and all $\lambda\in(0,\Lambda]$ the map $t \mapsto f(x,t,\lambda)+\sigma t^{p-1}$ is nondecreasing in $[0,T]$;
\item\label{h6} for a.e.\ $x\in\Omega$ and all $t>0$ the map $\lambda \mapsto f(x,t,\lambda)$ is increasing in $\R^+_0$;
\item\label{h7} for all $0<T_1<T_2$, uniformly for a.e.\ $x\in\Omega$ and all $t\in[T_1,T_2]$
\[\lim_{\lambda\to\infty} f(x,t,\lambda) = \infty.\]
\end{enumroman}
\end{itemize}
Hypothesis \ref{h1} is a subcritical growth condition. Hypotheses \ref{h2}, \ref{h3} govern the behavior of $f(x,\cdot,\lambda)$ at infinity, which is $(p-1)$-superlinear but tempered by an asymptotic condition of Ambrosetti-Rabinowitz type. By \ref{h4} $f(x,\cdot,\lambda)$ is $(p-1)$-sublinear near the origin, while \ref{h5} is a quasi-monotonicity condition. Finally, hypotheses \ref{h6}, \ref{h7} are related to the $\lambda$-dependence of the reaction. For some examples of functions satisfying ${\bf H}$, see the end of Section \ref{sec3}.
\vskip2pt
\noindent
Under hypotheses ${\bf H}$ we prove the following bifurcation-type result:

\begin{theorem}\label{bif}
Let ${\bf H}$ hold. Then, there exists $\lambda^*>0$ s.t.\
\begin{enumroman}
\item\label{bif1} for all $\lambda\in(0,\lambda^*)$ $(P_\lambda)$ has at least two solutions $0<u_\lambda<v_\lambda$, s.t.\ $u_\lambda<u_\mu$ for all $0<\lambda<\mu<\lambda^*$;
\item\label{bif2} $(P_{\lambda^*})$ has at least one solution $u^*>0$ s.t.\ $u_\lambda\to u^*$ uniformly in $\Omega$ as $\lambda\to\lambda^*$;
\item\label{bif3} for all $\lambda>\lambda^*$ $(P_\lambda)$ has no solutions.
\end{enumroman}
\end{theorem}

\noindent
See Section \ref{sec2} below for a proper definition of solution. Note that our result is new even in the semilinear case $p=2$ (fractional Laplacian). Also, note that we have no precise information on the behavior of the greater solution $v_\lambda$ as $\lambda\to\lambda^*$ (this is why Theorem \ref{bif} is not literally a bifurcation result).
\vskip4pt
\noindent
{\bf Notation:} Throughout the paper, for any $A\subset\R^N$ we shall set $A^c=\R^N\setminus A$. For any two measurable functions $u,v:\Omega\to\R$, $u\le v$ in $\Omega$ will mean that $u(x)\le v(x)$ for a.e.\ $x\in\Omega$ (and similar expressions). The positive (resp., negative) part of $u$ is denoted $u^+$ (resp., $u^-$). If $X$ is an ordered Banach space, then $X_+$ will denote its non-negative order cone. For all $r\in[1,\infty]$, $\|\cdot\|_r$ denotes the standard norm of $L^r(\Omega)$ (or $L^r(\R^N)$, which will be clear from the context). Every function $u$ defined in $\Omega$ will be identified with its $0$-extension to $\R^N$. Moreover, $C$ will denote a positive constant (whose value may change case by case).

\section{Preliminaries}\label{sec2}

\noindent
In this section we recall some basic theory on the Dirichlet problem for fractional $p$-Laplacian equation. We shall focus on such results which are most needed in our study and focus on simpler (if not most general) statements. We refer to \cite{ILPS} for a general introduction to variational methods for such problem, and to \cite{P} for a detailed account on related regularity theory.
\vskip2pt
\noindent
For all measurable $u:\Omega\to\R$, $s\in(0,1)$, $p>1$ we denote
\[[u]_{s,p,\Omega} = \Big[\iint_{\Omega\times\Omega}\frac{|u(x)-u(y)|^p}{|x-y|^{N+ps}}\,dx\,dy\Big]^\frac{1}{p}.\]
Accordingly we define the fractional Sobolev space
\[W^{s,p}(\Omega) = \big\{u\in L^p(\Omega):\,[u]_{s,p,\Omega}<\infty\big\}.\]
If $\Omega\subset\R^N$ is a bounded domain with $C^{1,1}$-boundary we also define
\[\w = \big\{u\in W^{s,p}(\R^N):\,u=0 \ \text{in $\Omega^c$}\big\},\]
a uniformly convex, separable Banach space with norm $\|u\|=[u]_{s,p,\R^N}$ and dual space $W^{-s,p'}(\Omega)$. Assume now that $ps<N$ and set
\[p^*_s = \frac{Np}{N-ps},\]
then the embedding $\w\hookrightarrow L^q(\Omega)$ is continuous for all $q\in[1,p^*_s]$ and compact for all $q\in[1,p^*_s)$ (see \cite{DPV} for a quick introduction to fractional Sobolev spaces). We can now extend the definition of the fractional $p$-Laplacian (of order $s$) by setting for all $u,\varphi\in\w$
\[\langle\fpl u,\varphi\rangle = \iint_{\R^N \times \R^N} \frac{|u(x)-u(y)|^{p-2} (u(x)-u(y)) (\varphi(x)-\varphi(y))}{|x-y|^{N+ps}}\,dx\,dy.\]
Such definition is equivalent to the one given in Section \ref{sec1}, provided $u$ is smooth enough, for instance if $u\in C^\infty_c(\Omega)$. More generally, we have defined $\fpl:\w\to W^{-s,p'}(\Omega)$ as a continuous, maximal monotone operator of $(S)_+$-type, i.e., whenever $u_n\rightharpoonup u$ in $\w$ and
\[\limsup_n \langle\fpl u_n,u_n-u\rangle \le 0,\]
then we have $u_n\to u$ (strongly) in $\w$. Also, the map $\fpl$ is strictly $(T)$-monotone, i.e., for all $u,v\in\w$ s.t.\
\[\langle \fpl u - \fpl v, (u-v)^+\rangle \le 0,\]
we have $u\le v$ in $\Omega$. Finally we recall that for all $u\in\w$
\[\|u^\pm\|^p \le \langle\fpl u,\pm u^\pm\rangle.\]
All these properties are proved (in a slightly more general form) in \cite{FI}, though some go back to previous works.
\vskip2pt
\noindent
The general Dirichlet problem for the fractional $p$-Laplacian is stated as follows:
\beq\label{dir}
\begin{cases}
\fpl u = f_0(x,u) & \text{in $\Omega$} \\
u = 0 & \text{in $\Omega^c$.}
\end{cases}
\eeq
The reaction $f_0$ is subject to the following basic hypotheses:
\begin{itemize}[leftmargin=1cm]
\item[${\bf H}_0$] $f_0:\Omega\times\R\to\R$ is a Carath\'eodory map and there exist $c_0>0$, $r\in(1,p^*_s)$ s.t.\ for a.e.\ $x\in\Omega$ and all $t\in\R$
\[|f_0(x,t)| \le c_0(1+|t|^{r-1}).\]
\end{itemize}
By virtue of ${\bf H}_0$ the following definitions are well posed. We say that $u\in\w$ is a (weak) supersolution of \eqref{dir} if for all $\varphi\in\w_+$
\[\langle\fpl u,\varphi\rangle \ge \int_\Omega f_0(x,u)\varphi\,dx.\]
The definition of a (weak) subsolution is analogous. Once again we remark that these are not the most general definitions of super- and subsolution, as in general one can require $u\ge 0$ or $u\le 0$, respectively, in $\Omega^c$ (see \cite{FI}). Finally, we say that $u\in\w$ is a (weak) solution of \eqref{dir} if it is both a super- and a subsolution, i.e., if for all $\varphi\in\w$
\[\langle\fpl u,\varphi\rangle = \int_\Omega f_0(x,u)\varphi\,dx.\]
For the solutions of \eqref{dir} we have the following a priori bound:

\begin{proposition}\label{apb}
{\rm\cite[Theorem 3.3]{CMS}} Let ${\bf H}_0$ hold, $u \in \w$ be a solution of \eqref{dir}. Then, $u \in L^{\infty}(\Omega)$ with $\|u\|_{\infty} \le C$, for some $C=C(\|u\|)>0$.
\end{proposition}

\noindent
Regularity of solutions to nonlocal equations is a delicate issue, as such solutions fail in general to be smooth up to the boundary of the domain (no matter how regular it is). Such problem can be overcome by comparing the solutions to a convenient power of the distance from the boundary, namely, set for all $x\in\R^N$
\[\ds(x)= \mathrm{dist}(x, \Omega^c)^s.\]
We define the following weighted H\"older spaces and the respective norms:
\[\cs = \Big\{u \in C^0(\overline\Omega): \frac{u}{\ds} \ \text{has a continuous extension to} \ \overline\Omega \Big\}, \ \|u\|_{0,s} = \Big\|\frac{u}{\ds}\Big\|_\infty,\]
and for all $\alpha\in(0,1)$
\[C_s^{\alpha}(\overline\Omega)= \Big\{u \in C^0(\overline\Omega): \frac{u}{\ds} \ \text{has a $\alpha$-H\"older continuous extension to} \ \overline\Omega\Big\}, \ \|u\|_{\alpha,s} = \Big\|\frac{u}{\ds}\Big\|_{C^\alpha(\overline\Omega)}.\]
The embedding $C_s^{\alpha}(\overline\Omega) \hookrightarrow \cs$ is compact for all $\alpha \in (0,1)$. Also, the positive cone $\cs_+$ of $\cs$ has a nonempty interior given by
\[{\rm int}(\cs_+)= \Big\{u \in \cs:\, \inf_{\Omega}\frac{u}{\ds} > 0\Big\}.\]
Combining Proposition \ref{apb} and \cite[Theorem 1.1]{IMS}, we have the following global regularity result for the degenerate case $p\ge 2$:

\begin{proposition}\label{reg}
Let $p\ge 2$, ${\bf H}_0$ hold, $u \in \w$ be a solution of \eqref{dir}. Then, $u \in C_s^{\alpha}(\overline\Omega)$  for some $\alpha \in (0,s]$, with $\|u\|_{\alpha,s}\le C(\|u\|)$.
\end{proposition}

\noindent
We recall now two recent strong maximum and comparison principles, which will be used in our study:

\begin{proposition}\label{smp}
{\rm\cite[Theorem 2.6]{IMP}} Let $g \in C^0(\R) \cap BV_{\rm loc}(\R)$, $u \in \w \cap C^0(\overline\Omega)\setminus\{0\}$ s.t.\ 
\[\begin{cases}
\fpl u +g(u) \ge g(0) & \text{weakly in $\Omega$} \\
u\ge0 & \text{in $\Omega$.}
\end{cases}\]
Then, 
\[\inf_{\Omega} \frac{u}{\ds}>0.\]
In particular, if $u \in \cs$, then $u \in {\rm int}(\cs_+)$. 
\end{proposition}

\begin{proposition}\label{scp}
{\rm\cite[Theorem 2.7]{IMP}} Let $g \in C^0(\R) \cap BV_{\rm loc}(\R)$, $u, v \in \w \cap C^0(\overline\Omega)$ s.t.\ $u \not\equiv v$, $C>0$ satisfy
\[\begin{cases}
\fpl v +g(v) \le \fpl u +g(u) \le C & \text{weakly in $\Omega$} \\
0 < v\le u & \text{in $\Omega$}. 
\end{cases}\]
Then, $u>v$ in $\Omega$. In particular, if $u,v\in{\rm int}(\cs_+)$, then $u-v\in{\rm int}(\cs_+)$.
\end{proposition}

\noindent
Next, we recall some properties related to the following nonlocal, nonlinear eigenvalue problem:
\beq\label{evp}
\begin{cases}
\fpl u = \hat\lambda|u|^{p-1}u & \text{in $\Omega$} \\
u = 0 & \text{in $\Omega^c$.}
\end{cases}
\eeq
Problem \eqref{evp} admits an unbounded sequence of variational (Lusternik-Schnirelmann) eigenvalues $(\hat\lambda_n)$. In particular, we focus on the principal eigenvalue $\hat\lambda_1$:

\begin{proposition}\label{pev}
{\rm\cite[Theorem 6]{LL} \cite[Theorems 4.1, 4.2]{FP}} The smallest eigenvalue of \eqref{evp} is
\[\hat\lambda_1 = \min_{u\in\w\setminus\{0\}}\frac{\|u\|^p}{\|u\|_p^p} > 0,\]
it is simple, isolated, and attained at a unique positive eigenfunction $e_1\in{\rm int}(\cs_+)$ s.t.\ $\|e_1\|_p=1$.
\end{proposition}

\noindent
All the non-principal eigenfunctions of \eqref{evp} are nodal (i.e., sign-changing) in $\Omega$. More generally, we have the following anti-maximum principle for the degenerate case:

\begin{proposition}\label{amp}
{\rm\cite[Lemma 3.9]{FI1}} Let $p\ge 2$, $\lambda\ge\hat\lambda_1$, $\beta\in L^\infty(\Omega)_+\setminus\{0\}$, and $u\in\w$ be a solution of
\[\begin{cases}
\fpl u = \lambda|u|^{p-2}u+\beta(x) & \text{in $\Omega$} \\
u = 0 & \text{in $\Omega^c$.}
\end{cases}\]
Then, $u^-\neq 0$.
\end{proposition}

\noindent
Finally, we introduce a variational framework for problem \eqref{dir}. For all $(x,t)\in\Omega\times\R$ set
\[F_0(x,t) = \int_0^t f_0(x,\tau)\,d\tau,\]
and for all $u\in\w$ set
\[\Phi_0(u) = \frac{\|u\|^p}{p}-\int_\Omega F_0(x,u)\,dx.\]
Then $\Phi_0\in C^1(\w)$ and its critical points coincide with the solutions of \eqref{dir}. Besides, $\Phi_0$ is sequentially weakly l.s.c.\ in $\w$ and its local minimizers in the topologies of $\w$ and $\cs$, respectively, coincide:

 \begin{proposition}\label{svh}
{\rm \cite[Theorem 1.1]{IMS1}} Let $p\ge 2$, ${\bf H}_0$ hold, and $u\in\w$. Then, the following are equivalent:
\begin{enumroman}
\item\label{svh1} there exists $\rho>0$ s.t.\ $\Phi_0(u+v)\ge\Phi_0(u)$ for all $v\in\w\cap C_s^0(\overline\Omega)$, $\|v\|_{0,s}\le\rho$;
\item\label{svh2} there exists $\sigma>0$ s.t.\ $\Phi_0(u+v)\ge\Phi_0(u)$ for all $v\in\w$, $\|v\| \le\sigma$.
\end{enumroman}
\end{proposition}

\section{Bifurcation-type result}\label{sec3}

\noindent
This section is devoted to the proof of Theorem \ref{bif}, which we split in several lemmas. We recall that $p\ge 2$, $s\in(0,1)$ satisfy $ps<N$, that $\Omega\subset\R^N$ is a bounded domain with $C^{1,1}$-boundary, and that the reaction $f$ in problem $(P_\lambda)$ satisfies the standing hypotheses ${\bf H}$ (for simplicity we shall omit such assumptions in the results of this section). Since ${\bf H}$ only deal with $t\ge 0$, without loss of generality we set for all $(x,t,\lambda)\in\Omega\times\R^-\times\R^+_0$
\[f(x,t,\lambda) = 0.\]
We note that, by ${\bf H}$ \ref{h1}, $f(\cdot,\cdot,\lambda):\Omega\times\R\to\R$ satisfies ${\bf H}_0$ for all $\lambda>0$. For all $\lambda>0$, $u\in\w$ we define the energy functional of $(P_\lambda)$
\[\Phi_\lambda(u) = \frac{\|u\|^p}{p}-\int_\Omega F(x,u,\lambda)\,dx.\]
We begin with a sub-supersolution principle:

\begin{lemma}\label{ssp}
Let $\lambda>0$, $\bar{u}\in{\rm int}(\cs_+)$ be a supersolution of $(P_\lambda)$. Then, there exists a solution $u\in{\rm int}(\cs_+)$ of $(P_\lambda)$ s.t.\ $u\le\bar{u}$ in $\Omega$.
\end{lemma}
\begin{proof}
We perform a truncation on the reaction, setting for all $(x,t)\in\Omega\times\R$
\[\bar{f}_\lambda(x,t) = \begin{cases}
f(x,t,\lambda) & \text{if $t<\bar{u}(x)$} \\
f(x,\bar{u}(x),\lambda) & \text{if $t\ge\bar{u}(x)$}
\end{cases}\]
and
\[\bar{F}_\lambda(x,t) = \int_0^t \bar{f}_\lambda(x,\tau)\,d\tau.\]
By ${\bf H}$ \ref{h1}, $\bar{f}_\lambda:\Omega\times\R\to\R$ satisfies ${\bf H}_0$. Moreover, for a.e.\ $x\in\Omega$ and all $t>\bar{u}(x)$ we have
\begin{align*}
\bar{F}_\lambda(x,t) &= \int_0^{\bar{u}} f(x,\tau,\lambda)\,d\tau+\int_{\bar{u}}^t f(x,\bar{u},\lambda)\,dx \\
&\le \int_0^{\bar{u}}\big(a_\lambda(x)+c_1\tau^{r-1}\big)\,d\tau+\int_{\bar{u}}^t\big(a_\lambda(x)+c_1\bar{u}^{r-1}\big)\,d\tau \le C(1+t).
\end{align*}
More generally, for a.e.\ $x\in\Omega$ and all $t\in\R$
\beq\label{ssp1}
\bar{F}_\lambda(x,t) \le C(1+|t|).
\eeq
Also we set for all $u\in\w$
\[\bar{\Phi}_\lambda(u) = \frac{\|u\|^p}{p}-\int_\Omega \bar{F}_\lambda(x,u)\,dx.\]
By ${\bf H}$ \ref{h1}, $\bar\Phi_\lambda\in C^1(\w)$ is sequentially weakly l.s.c. Besides, by \eqref{ssp1} and the continuous embedding $\w\hookrightarrow L^1(\Omega)$ we have for all $u\in\w$
\begin{align*}
\bar\Phi_\lambda(x) &\ge \frac{\|u\|^p}{p}-\int_\Omega C(1+|u|)\,dx \\
&\ge \frac{\|u\|^p}{p}-C(1+\|u\|),
\end{align*}
and the latter tends to $\infty$ as $\|u\|\to\infty$. So, $\bar\Phi_\lambda$ is coercive in $\w$. Thus, there exists $u\in\w$ s.t.\
\[\bar\Phi_\lambda(u) = \inf_{\w}\bar\Phi_\lambda = \bar{m}_\lambda.\]
Now let $e_1\in{\rm int}(\cs_+)$ be defined as in Proposition \ref{pev}, $\delta>0$ be as in ${\bf H}$ \ref{h4}. Then we can find $\tau>0$ s.t.\ in $\Omega$
\[0 < \tau e_1 < \min\{\bar{u},\,\delta\}.\]
By ${\bf H}$ \ref{h4} and the construction of $\bar\Phi_\lambda$ we have
\begin{align*}
\bar\Phi_\lambda(\tau e_1) &= \frac{\tau^p\|e_1\|^p}{p}-\int_\Omega F(x,\tau e_1,\lambda)\,dx \\
&\le \frac{\tau^p\hat\lambda_1}{p}-\frac{\tau^q c_2\|e_1\|_q^q}{q},
\end{align*}
and the latter is negative for $\tau>0$ even smaller if necessary (by $q<p$). So $\bar{m}_\lambda<0$, which in turn implies $u_\lambda\neq 0$. By minimization we have weakly in $\Omega$
\beq\label{ssp2}
\fpl u = \bar{f}_\lambda(x,u).
\eeq
By Proposition \ref{reg} we have $u\in C^\alpha_s(\overline\Omega)$. Testing \eqref{ssp2} with $-u^-\in\w$ we have
\begin{align*}
\|u^-\|^p &\le \langle\fpl u,-u^-\rangle \\
&= \int_{\{u<0\}} f(x,u,\lambda)u\,dx = 0.
\end{align*}
So $u\ge 0$ in $\Omega$. On the other hand, testing \eqref{ssp2} with $(u-\bar{u})^+\in\w$ and recalling that $\bar{u}$ is a supersolution of $(P_\lambda)$, we have
\[\langle\fpl u-\fpl\bar{u},(u-\bar{u})^+\rangle \le \int_{\{u>\bar{u}\}} \big(\bar{f}_\lambda(x,u)-f(x,\bar{u},\lambda)\big)(u-\bar{u})\,dx = 0.\]
By strict $(T)$-monotonicity of $\fpl$ we have $u\le\bar{u}$ in $\Omega$. By construction, then, we can rephrase \eqref{ssp2} and have weakly in $\Omega$
\[\fpl u = f(x,u,\lambda).\]
By ${\bf H}$ \ref{h5} (with $T=\|\bar{u}\|_\infty$ and $\Lambda=\lambda$) there exists $\sigma>0$ s.t.\ for a.e.\ $x\in\Omega$ the mapping
\[t \mapsto f(x,t,\lambda)+\sigma t^{p-1}\]
is nondecreasing in $[0,\|\bar{u}\|_\infty]$. So weakly in $\Omega$
\[\fpl u+\sigma u^{p-1} = f(x,u,\lambda)+\sigma u^{p-1} \ge 0.\]
By Proposition \ref{smp} (with $g(t)=\sigma(t^+)^{p-1}$) we have $u\in{\rm int}(\cs_+)$, in particular $u>0$ in $\Omega$, so $u$ solves $(P_\lambda)$.
\end{proof}

\noindent
Set
\beq\label{str}
\lambda^* = \sup\big\{\lambda>0:\,(P_\lambda) \ \text{has a solution $u_\lambda\in{\rm int}(\cs_+)$}\big\}
\eeq
(with the convention $\inf\,\emptyset=\infty$). We will now establish some properties of $\lambda^*$:

\begin{lemma}\label{exi}
Let $\lambda^*$ be defined by \eqref{str}. Then we have
\begin{enumroman}
\item\label{exi1} $0<\lambda^*<\infty$;
\item\label{exi2} for all $\lambda\in(0,\lambda^*)$ $(P_\lambda)$ has a solution $u_\lambda\in{\rm int}(\cs_+)$;
\item\label{exi3} for all $\lambda,\mu\in(0,\lambda^*)$ s.t.\ $\lambda<\mu$ we have $u_\mu-u_\lambda\in{\rm int}(\cs_+)$.
\end{enumroman}
\end{lemma}
\begin{proof}
First we consider the auxiliary problem (torsion equation)
\beq\label{exi4}
\begin{cases}
\fpl w = 1 & \text{in $\Omega$} \\
w = 0 & \text{in $\Omega^c$.}
\end{cases}
\eeq
The corresponding energy functional $\Psi\in C^1(\w)$ is defined for all $w\in\w$ by
\[\Psi(w) = \frac{\|w\|^p}{p}-\int_\Omega w\,dx.\]
As in Section \ref{sec2} we see that $\Psi$ is coercive and sequentially weakly l.s.c., so there exists $w\in\w$ s.t.\
\[\Psi(w) = \inf_{\w}\Psi.\]
In particular, $w$ is a critical point of $\Phi$ and hence a solution of \eqref{exi4}, so by Proposition \ref{reg} we have $w\in C^\alpha_s(\overline\Omega)$. Testing \eqref{exi4} with $-w^-\in\w$ we get
\begin{align*}
\|w^-\|^p &\le \langle\fpl w,-w^-\rangle \\
&= \int_{\{w<0\}}w\,dx \le 0,
\end{align*}
so $w\ge 0$ in $\Omega$. Also, clearly $w\neq 0$. By Proposition \ref{smp}, then, we have $w\in{\rm int}(\cs_+)$.
\vskip2pt
\noindent
Now we prove \ref{exi1}. First, we claim that there exists $\check\lambda>0$ with the following property: for all $\lambda\in(0,\check\lambda)$ there is $\xi_\lambda>0$ s.t.\
\beq\label{exi5}
\|a_\lambda\|_\infty+c_1(\xi_\lambda\|w\|_\infty)^{r-1} < \xi_\lambda^{p-1}
\eeq
(with $a_\lambda\in L^\infty(\Omega)_+$, $c_1>0$ as in ${\bf H}$ \ref{h1}). Arguing by contradiction, let $(\lambda_n)$ be a sequence s.t.\ $\lambda_n\to 0^+$ and for all $n\in\N$, $\xi>0$
\[\|a_{\lambda_n}\|_\infty+c_1(\xi\|w\|_\infty)^{r-1} \ge \xi^{p-1}.\]
By ${\bf H}$ \ref{h1} we have $\|a_{\lambda_n}\|_\infty\to 0$, so passing to the limit as $n\to\infty$ we get for all $\xi>0$
\[c_1\|w\|_\infty^{r-1} > \xi^{p-r},\]
which yields a contradiction as $\xi\to 0^+$. We prove next that $\lambda^*\ge\check\lambda$. Indeed, for all $\lambda\in(0,\check\lambda)$ let $\xi_\lambda>0$ satisfy \eqref{exi5}, and set
\[\bar{u} = \xi_\lambda w \in {\rm int}(\cs_+).\]
By \eqref{exi4}, \eqref{exi5}, and ${\bf H}$ \ref{h1} we have weakly in $\Omega$
\[\fpl\bar{u} = \xi_\lambda^{p-1} \ge a_\lambda(x)+c_1\bar{u}^{r-1} \ge f(x,\bar{u},\lambda),\]
i.e., $\bar{u}$ is a (strict) supersolution of $(P_\lambda)$. By Lemma \ref{ssp} there exists a solution $u_\lambda\in{\rm int}(\cs_+)$ of $(P_\lambda)$ s.t.\ $u_\lambda\le\bar{u}$ in $\Omega$. Hence we have $\lambda^*\ge\lambda$. Taking the supremum over $\lambda$ we get as claimed
\[\lambda^* \ge \check\lambda > 0.\]
Looking on the opposite side, we claim that there exists $\hat\lambda>0$ s.t.\ for all $\lambda\ge\hat\lambda$ we have for a.e.\ $x\in\Omega$ and all $t>0$
\beq\label{exi6}
f(x,t,\lambda) > \hat\lambda_1 t^{p-1}
\eeq
(with $\hat\lambda_1>0$ as in Lemma \ref{pev}). Indeed, by ${\bf H}$ \ref{h2}, given $\lambda=1$ we can find $T_2>0$ s.t.\ for a.e.\ $x\in\Omega$ and all $t>T_2$
\[f(x,t,1) > \hat\lambda_1 t^{p-1}.\]
By ${\bf H}$ \ref{h6}, for all $\lambda\ge 1$, a.e.\ $x\in\Omega$, and all $t>T_2$ we have
\[f(x,t,\lambda) > \hat\lambda_1 t^{p-1}.\]
Besides, let $c_2,\delta>0$, $q\in(1,p)$ be as in ${\bf H}$ \ref{h4}. Then we can find $T_1>0$ s.t.\
\[T_1 < \min\{T_2,\,\delta\}, \ \frac{c_2}{T_1^{p-q}} > \hat\lambda_1.\]
Hence, for all $\lambda\ge 1$, a.e.\ $x\in\Omega$ and $t\in(0,T_1)$ we have by ${\bf H}$ \ref{h4}
\[f(x,t,\lambda) \ge c_2t^{q-1} > \frac{c_2}{T_1^{p-q}}t^{p-1} \ge \hat\lambda_1 t^{p-1}.\]
Finally, by ${\bf H}$ \ref{h7} we have uniformly for a.e.\ $x\in\Omega$ and all $t\in[T_1,T_2]$
\[\lim_{\lambda\to\infty} f(x,t,\lambda) = \infty,\]
so we can find $\hat\lambda\ge 1$ s.t.\ for all $\lambda\ge\hat\lambda$, a.e.\ $x\in\Omega$, and all $t\in[T_1,T_2]$
\[f(x,t,\lambda) > \hat\lambda_1 T_2^{p-1} \ge \hat\lambda_1 t^{p-1}.\]
Putting the inequalities above in a row, we get \eqref{exi6}. We see that $\lambda^*\le\hat\lambda$, arguing by contradiction. Let $\lambda>\hat\lambda$ be s.t.\ $(P_\lambda)$ has a solution $u_\lambda\in{\rm int}(\cs_+)$. Then by \eqref{exi6} we have weakly in $\Omega$
\[\fpl u_\lambda = f(x,u_\lambda,\lambda) > \hat\lambda_1 u_\lambda^{p-1}.\]
Set for all $x\in\Omega$
\[\beta(x) = f(x,u_\lambda(x),\lambda)-\hat\lambda_1 u_\lambda(x)^{p-1},\]
then by ${\bf H}$ \ref{h1} and the inequality above we have $\beta\in L^\infty(\Omega)_+$, $\beta\neq 0$. By Proposition \ref{amp} we have $u_\lambda^-\neq 0$, a contradiction. Thus we have
\[\lambda^* \le \hat\lambda < \infty.\]
Further, we prove \ref{exi2}. For all $\lambda\in(0,\lambda^*)$ we can find $\mu\in(\lambda,\lambda^*)$ s.t.\ $(P_\mu)$ has a solution $u_\mu\in{\rm int}(\cs_+)$. By ${\bf H}$ \ref{h6} we have weakly in $\Omega$
\[\fpl u_\mu = f(x,u_\mu,\mu) > f(x,u_\mu,\lambda),\]
i.e., $u_\mu$ is a (strict) supersolution of $(P_\lambda)$. By Lemma \ref{ssp} there exists a solution $u_\lambda\in{\rm int}(\cs_+)$ of $(P_\lambda)$ s.t.\ $u_\lambda\le u_\mu$ in $\Omega$.
\vskip2pt
\noindent
Finally, we prove \ref{exi3}. For all $0<\lambda<\mu<\lambda^*$, reasoning as above we find $u_\lambda,u_\mu\in{\rm int}(\cs_+)$ solutions of $(P_\lambda)$, $(P_\mu)$ respectively, s.t.\ $u_\lambda\le u_\mu$ in $\Omega$. Invoking ${\bf H}$ \ref{h5} (with $T=\|u_\mu\|_\infty$ and $\Lambda=\lambda^*$), we find $\sigma>0$ s.t.\ the mapping
\[t \mapsto f(x,t,\lambda)+\sigma t^{p-1}\]
is nondecreasing in $[0,\|u_\mu\|_\infty]$. So, using also ${\bf H}$ \ref{h6}, weakly in $\Omega$ we have
\begin{align*}
\fpl u_\lambda+\sigma u_\lambda^{p-1} &= f(x,u_\lambda,\lambda)+\sigma u_\lambda^{p-1} \\
&\le f(x,u_\mu,\lambda)+\sigma u_\mu^{p-1} \\
&< f(x,u_\mu,\mu)+\sigma u_\mu^{p-1} \\
&= \fpl u_\mu+\sigma u_\mu^{p-1}.
\end{align*}
By Proposition \ref{scp} (with $g(t)=\sigma(t^+)^{p-1}$) we have $u_\mu-u_\lambda\in{\rm int}(\cs_+)$.
\end{proof}

\noindent
In the next result we deal with the threshold case $\lambda=\lambda^*$:

\begin{lemma}\label{trs}
Let $\lambda^*>0$ be defined by \eqref{str}. Then $(P_{\lambda^*})$ has at least one solution $u^*\in{\rm int}(\cs_+)$.
\end{lemma}
\begin{proof}
Let $(\lambda_n)$ be an increasing sequence in $\R^+_0$ s.t.\ $\lambda_n \to\lambda^*$. Recalling the proof of Lemma \ref{exi} \ref{exi2}, we know that for all $n\in\N$ problem $(P_{\lambda_n})$ has a solution $u_n\in{\rm int}(\cs_+)$ with negative energy, i.e., weakly in $\Omega$
\beq\label{trs1}
\fpl u_n = f(x,u_n,\lambda_n)
\eeq
and
\beq\label{trs2}
\|u_n\|^p-p\int_\Omega F(x,u_n,\lambda_n)\,dx < 0.
\eeq
Also, by Lemma \ref{exi} \ref{exi3} we have $u_n<u_m$ in $\Omega$ for all $n<m$. Testing \eqref{trs1} with $u_n\in\w$ we get
\[\|u_n\|^p = \int_\Omega f(x,u_n,\lambda_n)u_n\,dx,\]
which along with \eqref{trs2} gives
\beq\label{trs3}
\int_\Omega\big(f(x,u_n,\lambda_n)u_n-pF(x,u_n,\lambda_n)\big)\,dx < 0.
\eeq
By ${\bf H}$ \ref{h3} (with $\Lambda=\lambda^*$) there exist $\theta,T>0$ s.t.\ for all $n\in\N$, a.e.\ $x\in\Omega$, and all $t\ge T$ we have
\[f(x,t,\lambda_n)t-pF(x,t,\lambda_n) \ge \theta t^\rho.\]
Also, by ${\bf H}$ \ref{h1} we can find $C>0$ s.t.\ for all $n\in\N$, a.e.\ $x\in\Omega$ and all $t\in[0,T]$
\[\big|f(x,t,\lambda_n)t-pF(x,t,\lambda_n)\big| \le C.\]
Summarizing, we have for all $n\in\N$, a.e.\ $x\in\Omega$, and all $t\ge 0$
\[f(x,t,\lambda_n)t-pF(x,t,\lambda_n) \ge \theta t^\rho-C,\]
with $\theta,C>0$ independent of $n$. Plugging the estimate above into \eqref{trs3} we get for all $n\in\N$
\[0 > \int_\Omega\big(\theta u_n^\rho-C\big)\,dx = \theta\|u_n\|_\rho^\rho-C.\]
So, $(u_n)$ is a bounded sequence in $L^\rho(\Omega)$. Since $\rho\le r<p^*_s$, we can find $\tau\in[0,1)$ s.t.\
\[\frac{1}{r} = \frac{1-\tau}{\rho}+\frac{\tau}{p^*_s}.\]
By the interpolation and Sobolev's inequalities, we have for all $n\in\N$
\[\|u_n\|_r \le \|u_n\|_\rho^{1-\tau}\|u_n\|_{p^*_s}^\tau \le C\|u_n\|^\tau.\]
A straightforward calculation leads from the bounds on $\rho$ in ${\bf H}$ \ref{h3} to $\tau r<p$. Now test \eqref{trs1} with $u_n\in\w$ again and use ${\bf H}$ \ref{h1} to get
\begin{align*}
\|u_n\|^p &= \int_\Omega f(x,u_n,\lambda_n)u_n\,dx \\
&\le \int_\Omega\big(a_{\lambda_n}(x)+c_1 u_n^{r-1}\big)u_n\,dx \\
&\le C\big(\|u_n\|_1+\|u_n\|_r^r\big) \\
&\le C\big(\|u_n\|+\|u_n\|^{\tau r}\big).
\end{align*}
So $(u_n)$ is bounded in $\w$. Passing to a subsequence if necessary, we find $u^*\in\w$ s.t.\ $u_n\rightharpoonup u^*$ in $\w$, $u_n\to u^*$ in $L^r(\Omega)$, and $u_n(x)\to u^*(x)$ for a.e.\ $x\in\Omega$. In particular, we have $u^*\ge 0$ in $\Omega$. Test now \eqref{trs1} with $(u_n-u^*)\in\w$, use ${\bf H}$ \ref{h1} and H\"older's inequality to get for all $n\in\N$
\begin{align*}
\langle\fpl u_n,u_n-u^*\rangle &= \int_\Omega f(x,u_n,\lambda_n)(u_n-u^*)\,dx \\
&\le C\int_\Omega (1+u_n^{r-1})(u_n-u^*)\,dx \\
&\le C(1+\|u_n\|_r^{r-1})\|u_n-u^*\|_r,
\end{align*}
and the latter tends to $0$ as $n\to\infty$. So we have
\[\limsup_n\langle\fpl u_n,u_n-u^*\rangle \le 0,\]
which by the $(S)_+$-property of $\fpl$ implies $u_n\to u^*$ in $\w$. So we can pass to the limit as $n\to\infty$ in \eqref{trs1} and see that weakly in $\Omega$
\[\fpl u^* = f(x,u^*,\lambda^*).\]
By Proposition \ref{reg} we have $u^*\in C^\alpha_s(\overline\Omega)$. Finally, since $(u_n)$ is pointwise increasing, we have
\[\inf_\Omega\frac{u^*}{\ds} \ge \inf_\Omega\frac{u_n}{\ds} > 0.\]
So, $u^*\in{\rm int}(\cs_+)$ is a solution of $(P_{\lambda^*})$.
\end{proof}

\noindent
Finally we prove that for any parameter below the threshold there exists a second solution. This is in fact a fairly technical step in our study, involving some typical variational tricks. In particular, we recall the following notion:

\begin{definition}\label{cer}
{\rm \cite[Definition 5.14]{MMP}} Let $(X,\|\cdot\|)$ be a Banach space, $\Phi\in C^1(X)$. $\Phi$ satisfies the Cerami $(C)$-condition if every sequence $(u_n)$ in $X$, s.t.\ $(\Phi(u_n))$ is bounded and $(1+\|u_n\|)\Phi'(u_n)\to 0$ in $X^*$, has a (strongly) convergent subsequence.
\end{definition}

\noindent
We can now prove our multiplicity result:

\begin{lemma}\label{mul}
Let $\lambda^*>0$ be defined by \eqref{str}. Then, for all $\lambda\in(0,\lambda^*)$ problem $(P_\lambda)$ has a solution $v_\lambda\in{\rm int}(\cs_+)$ s.t.\ $v_\lambda-u_\lambda\in{\rm int}(\cs_+)$.
\end{lemma}
\begin{proof}
From Lemma \ref{trs} we know that $(P_{\lambda^*})$ has a solution $u^*\in{\rm int}(\cs_+)$. Now fix $\lambda\in(0,\lambda^*)$. By ${\bf H}$ \ref{h6} we have weakly in $\Omega$
\[\fpl u^* = f(x,u^*,\lambda^*) > f(x,u^*,\lambda),\]
so $u^*$ is a strict supersolution of $(P_\lambda)$. By Lemma \ref{ssp}, we see that $(P_\lambda)$ has a solution $u_\lambda\in{\rm int}(\cs_+)$ s.t.\ $u_\lambda\le u^*$ in $\Omega$ (without any loss of generality we may assume that such $u_\lambda$ is the same as in Lemma \ref{exi}). By ${\bf H}$ \ref{h5} (with $T=\|u^*\|_\infty$ and $\Lambda=\lambda^*$) there exists $\sigma>0$ s.t.\ for a.e.\ $x\in\Omega$ the mapping
\[t \mapsto f(x,t,\lambda)+\sigma t^{p-1}\]
is nondecreasing in $[0,\|u^*\|_\infty]$. By ${\bf H}$ \ref{h6} we have weakly in $\Omega$
\begin{align*}
\fpl u_\lambda+\sigma u_\lambda^{p-1} &= f(x,u_\lambda,\lambda)+\sigma u_\lambda^{p-1} \\
&\le f(x,u^*,\lambda)+\sigma (u^*)^{p-1} \\
&< f(x,u^*,\lambda^*)+\sigma (u^*)^{p-1} \\
&= \fpl u^*+\sigma(u^*)^{p-1}.
\end{align*}
By Proposition \ref{scp} (with $g(t)=\sigma (t^+)^{p-1}$) we have
\[u^*-u_\lambda \in {\rm int}(\cs_+).\]
Set for all $(x,t)\in\Omega\times\R$
\[\hat f_\lambda(x,t) = \begin{cases}
f(x,u_\lambda(x),\lambda) & \text{if $t\le u_\lambda(x)$} \\
f(x,t,\lambda) & \text{if $t>u_\lambda(x)$}
\end{cases}\]
and
\[\hat F_\lambda(x,t) = \int_0^t \hat f_\lambda(x,\tau)\,d\tau.\]
Also set for all $u\in\w$
\[\hat\Phi_\lambda(u) = \frac{\|u\|^p}{p}-\int_\Omega \hat F_\lambda(x,u)\,dx.\]
Clearly $\hat f_\lambda:\Omega\times\R\to\R$ satisfies ${\bf H}_0$, so $\hat\Phi_\lambda\in C^1(\w)$. The rest of the proof aims at showing the following claim:
\beq\label{mul1}
\text{$\hat\Phi_\lambda$ has a critical point $v_\lambda\in{\rm int}(\cs_+)$ s.t.\ $u_\lambda\le v_\lambda$ in $\Omega$, $u_\lambda\neq v_\lambda$.}
\eeq
We proceed by dichotomy. First, we introduce a new truncation of the reaction, setting for all $(x,t)\in\Omega\times\R$
\[\tilde f_\lambda(x,t) = \begin{cases}
f(x,u_\lambda(x),\lambda) & \text{if $t\le u_\lambda(x)$} \\
f(x,t,\lambda) & \text{if $u_\lambda(x)<t< u^*(x)$} \\
f(x,u^*(x),\lambda) & \text{if $t\ge u^*(x)$}
\end{cases}\]
and
\[\tilde F_\lambda(x,t) = \int_0^t \tilde f_\lambda(x,\tau)\,d\tau.\]
Since $u_\lambda,u^*\in{\rm int}(\cs_+)$, $\tilde f_\lambda:\Omega\times\R\to\R$ satisfies ${\bf H}_0$. Also, reasoning as in the proof of \eqref{ssp1} we see that for a.e.\ $x\in\Omega$ and all $t\in\R$ we have
\[|\tilde F_\lambda(x,t)| \le C(1+|t|).\]
Set for all $u\in\w$
\[\tilde\Phi_\lambda(u) = \frac{\|u\|^p}{p}-\int_\Omega \tilde F_\lambda(x,u)\,dx.\]
Then $\tilde\Phi_\lambda\in C^1(\w)$ is coercive and sequentially weakly l.s.c. So there exists $v_\lambda\in\w$ s.t.\
\beq\label{mul2}
\tilde\Phi_\lambda(v_\lambda) = \inf_{\w}\tilde\Phi_\lambda.
\eeq
In particular, we have weakly in $\Omega$
\beq\label{mul3}
\fpl v_\lambda = \tilde f_\lambda(x,v_\lambda).
\eeq
By Proposition \ref{reg} we have $v_\lambda\in C^\alpha_s(\overline\Omega)$. Testing $(P_\lambda)$ and \eqref{mul3} with $(u_\lambda-v_\lambda)^+\in\w$ we have
\[\langle\fpl u_\lambda-\fpl v_\lambda,(u_\lambda-v_\lambda)^+\rangle = \int_{\{u_\lambda>v_\lambda\}}\big(f(x,u_\lambda,\lambda)-\tilde f_\lambda(x,v_\lambda,\lambda)\big)(u_\lambda-v_\lambda)\,dx = 0,\]
so by strict $(T)$-monotonicity $u_\lambda\le v_\lambda$ in $\Omega$. As a consequence, $v_\lambda\in{\rm int}(\cs_+)$. Also, testing \eqref{mul3} and $(P_{\lambda^*})$ with $(v_\lambda-u^*)^+\in\w$, we get
\[\langle\fpl v_\lambda-\fpl u^*,(v_\lambda-u^*)^+\rangle = \int_{\{v_\lambda>u^*\}}\big(\tilde f_\lambda(x,v_\lambda)-f(x,u^*,\lambda^*)\big)(v_\lambda-u^*)\,dx,\]
and the latter is non-positive by $\lambda<\lambda^*$ and ${\bf H}$ \ref{h6}. So, as above $v_\lambda\le u^*$ in $\Omega$. Thus, in \eqref{mul3} we can replace $\tilde f_\lambda(x,v_\lambda)$ by $f(x,v_\lambda,\lambda)$ and see that $v_\lambda\ge u_\lambda$ is a critical point of $\hat\Phi_\lambda$.
\vskip2pt
\noindent
Now, either $v_\lambda\neq u_\lambda$, and then \eqref{mul1} is proved, or $v_\lambda=u_\lambda$, i.e., by \eqref{mul2} we have
\[\tilde\Phi_\lambda(u_\lambda) = \inf_{\w}\tilde\Phi_\lambda.\]
Set now
\[V = \big\{u^*-v:\,v\in{\rm int}(\cs_+)\big\},\]
an open set in the $\cs$-topology s.t.\ $u_\lambda\in V$. By construction, for all $u\in V$ we have
\[\hat\Phi_\lambda(u) = \tilde\Phi_\lambda(u) \ge \tilde\Phi_\lambda(u_\lambda) = \hat\Phi_\lambda(u_\lambda).\]
So, $u_\lambda$ is a local minimizer of $\hat\Phi_\lambda$ in $\cs$. By Proposition \ref{svh}, then, $u_\lambda$ is a local minimizer of $\hat\Phi_\lambda$ in $\w$ as well. Once again, an alternative shows: either there exists a critical point $v_\lambda\neq u_\lambda$ of $\hat\Phi_\lambda$, and as above we deduce $v_\lambda\ge u_\lambda$, hence \eqref{mul1} is proved; or $u_\lambda$ is a strict local minimizer of $\hat\Phi_\lambda$.
\vskip2pt
\noindent
We prove now that $u_\lambda$ is not a global minimizer of $\hat\Phi_\lambda$. Indeed, by ${\bf H}$ \ref{h2} and de l'H\^opital's rule we have uniformly for a.e.\ $x\in\Omega$
\[\lim_{t\to\infty}\frac{F(x,t,\lambda)}{t^p} = \infty.\]
Let $\hat\lambda_1>0$, $e_1\in{\rm int}(\cs_+)$ be defined as in Proposition \ref{pev}, and fix $\eps>0$. Then, we can find $T>0$ s.t.\ for a.e.\ $x\in\Omega$ and all $t\ge T$
\[F(x,t,\lambda) \ge \frac{\hat\lambda_1+\eps}{p}t^p.\]
By ${\bf H}$ \ref{h1} and construction of $\hat f_\lambda$, we can find $C>0$ s.t.\ for a.e.\ $x\in\Omega$ and all $t\ge 0$
\[\hat F_\lambda(x,t) \ge \frac{\hat\lambda_1+\eps}{p}t^p-C.\]
So, for all $\tau>0$ we have
\begin{align*}
\hat\Phi_\lambda(\tau e_1) &\le \frac{\tau^p\|e_1\|^p}{p}-\int_\Omega\Big(\frac{\hat\lambda_1+\eps}{p}(\tau e_1)^p-C\Big)\,dx \\
&\le \frac{\tau^p\hat\lambda_1}{p}-\frac{\tau^p(\hat\lambda_1+\eps)}{p}+C,
\end{align*}
and the latter tends to $-\infty$ as $\tau\to\infty$. So there exists $\tau>0$ s.t.\
\[\hat\Phi_\lambda(\tau e_1) < \hat\Phi_\lambda(u_\lambda).\]
In order to complete the geometrical picture, we deduce from the previous estimates that there exists $R\in(0,\|\tau e_1-u_\lambda\|)$ s.t.\
\[\inf_{\|u-u_\lambda\|=R}\hat\Phi_\lambda(u) = \eta \ge \hat\Phi_\lambda(u_\lambda) > \hat\Phi_\lambda(\tau e_1).\]
The next step consists in proving that $\hat\Phi_\lambda$ satisfies $(C)$ (see Definition \ref{cer} above). Let $(v_n)$ be a sequence in $\w$ s.t.\ $|\hat\Phi_\lambda(v_n)|\le C$ for all $n\in\N$, and $(1+\|v_n\|)\hat\Phi'_\lambda(v_n)\to 0$ in $W^{-s,p'}(\Omega)$ as $n\to\infty$. Then we have for all $n\in\N$
\[\Big|\|v_n\|^p-p\int_\Omega \hat F_\lambda(x,v_n)\,dx\Big| \le C,\]
and there exists a sequence $(\eps_n)$ s.t.\ $\eps_n\to 0^+$ and for all $n\in\N$, $\varphi\in\w$
\beq\label{mul4}
\Big|\langle\fpl v_n,\varphi\rangle-\int_\Omega \hat f_\lambda(x,v_n)\varphi\,dx\Big| \le \frac{\eps_n\|\varphi\|}{1+\|v_n\|}.
\eeq
Subtracting the inequalities above we get for all $n\in\N$
\beq\label{mul5}
\int_\Omega\big(\hat f_\lambda(x,v_n)v_n-p\hat F_\lambda(x,v_n)\big)\,dx \le C.
\eeq
By ${\bf H}$ \ref{h3} we can find $\theta,T>0$ s.t.\ for a.e.\ $x\in\Omega$ and all $t\ge T$
\[f(x,t,\lambda)t-pF(x,t,\lambda) \ge \theta t^\rho.\]
By ${\bf H}$ \ref{h1} and the construction of $\hat f_\lambda$, we can find $C>0$ s.t.\ for a.e.\ $x\in\Omega$ and all $t\ge 0$
\[\hat f_\lambda(x,t)t-p\hat F_\lambda(x,t) \ge \theta t^\rho-C.\]
Plugging such estimate into \eqref{mul5} we have for all $n\in\N$
\[\theta\|v_n\|_\rho^\rho \le \int_\Omega\big(\hat f_\lambda(x,v_n)v_n-p\hat F_\lambda(x,v_n)\big)\,dx+C \le C,\]
so $(v_n)$ is bounded in $L^\rho(\Omega)$. By the interpolation and Sobolev's inequalities, for all $n\in\N$ we have
\[\|v_n\|_r \le C\|v_n\|^\tau\]
for some $\tau\in[0,1)$ independent of $n\in\N$ s.t.\ $\tau r<p$ (see the proof of Lemma \ref{trs}). By \eqref{mul4} (with $\varphi=v_n$), ${\bf H}$ \ref{h1}, and H\"older's inequality, we have for all $n\in\N$
\begin{align*}
\|v_n\|^p &\le \int_\Omega \hat f_\lambda(x,v_n)v_n\,dx+\frac{\eps_n\|v_n\|}{1+\|v_n\|} \\
&\le C\int_\Omega (1+|v_n|^{r-1})|v_n|\,dx+\eps_n \\
&\le C\big(\|v_n\|_1+\|v_n\|_r^{\tau r}\big)+\eps_n.
\end{align*}
So $(v_n)$ is bounded in $\w$. Passing to a subsequence, we have $v_n\rightharpoonup v$ in $\w$ and $v_n\to v$ in both $L^1(\Omega)$ and $L^r(\Omega)$. Setting $\varphi=v_n-v$ in \eqref{mul4} and using ${\bf H}$ \ref{h1} and H\"older's inequality again, we have for all $n\in\N$
\begin{align*}
\langle\fpl v_n,v_n-v\rangle &\le \int_\Omega \hat f_\lambda(x,v_n)(v_n-v)\,dx+\frac{\eps\|v_n-v\|}{1+\|v_n\|} \\
&\le C\int_\Omega(1+|v_n|^{r-1})|v_n-v|\,dx+C\eps_n \\
&\le C\big(\|v_n-v\|_1+\|v_n\|_r^{r-1}\|v_n-v\|_r\big)+C\eps_n,
\end{align*}
and the latter tends to $0$ as $n\to\infty$. By the $(S)_+$-property of $\fpl$ we have $v_n\to v$ in $\w$, which proves $(C)$.
\vskip2pt
\noindent
We have now all the necessary ingredient to apply the mountain pass theorem (see for instance \cite[Theorem 5.40]{MMP}). Set
\[\Gamma = \big\{\gamma\in C([0,1],\w):\,\gamma(0)=u_\lambda,\,\gamma(1)=\tau e_1\big\}\]
and
\[c = \inf_{\gamma\in\Gamma}\max_{t\in[0,1]}\hat\Phi_\lambda(\gamma(t)).\]
Then $c\ge\eta$ and there exists a critical point $v_\lambda\in\w$ of $\hat\Phi_\lambda$ s.t.\ $\hat\Phi_\lambda(v_\lambda)=c$. Moreover, if $c=\eta$, then $\|v_\lambda-u_\lambda\|=R$. So $v_\lambda\neq u_\lambda$ satisfies weakly in $\Omega$
\[\fpl v_\lambda = \hat f_\lambda(x,v_\lambda).\]
As above we see that $v_\lambda\ge u_\lambda$, thus proving \eqref{mul1} in all cases.
\vskip2pt
\noindent
By construction of $\hat f_\lambda$, $v_\lambda$ solves $(P_\lambda)$, hence $v_\lambda\in C^\alpha_s(\overline\Omega)$. Also, from $v_\lambda\ge u_\lambda$ we deduce that $v_\lambda\in{\rm int}(\cs_+)$. By ${\bf H}$ \ref{h5} (with $T=\|v_\lambda\|_\infty$) there exists $\sigma>0$ s.t.\ for a.e.\ $x\in\Omega$ the mapping
\[t \mapsto f(x,t,\lambda)+\sigma t^{p-1}\]
is nondecreasing in $[0,\|v_\lambda\|_\infty]$. So we have weakly in $\Omega$
\begin{align*}
\fpl u_\lambda+\sigma u_\lambda^{p-1} &= f(x,u_\lambda,\lambda)+\sigma u_\lambda^{p-1} \\
&\le f(x,v_\lambda,\lambda)+\sigma v_\lambda^{p-1} \\
&= \fpl v_\lambda+\sigma v_\lambda^{p-1}.
\end{align*}
By Proposition \ref{scp} (with $g(t)=\sigma(t^+)^{p-1}$) we have $v_\lambda-u_\lambda\in{\rm int}(\cs_+)$.
\end{proof}

\noindent
We can now complete the proof of Theorem \ref{bif}:
\vskip2pt
\noindent
\begin{proof}
Let $\lambda^*>0$ be defined by \eqref{str}. By Lemmas \ref{exi} and \ref{mul}, for all $\lambda\in(0,\lambda^*)$ problem $(P_\lambda)$ has at least two solutions $u_\lambda,v_\lambda\in{\rm int}(\cs_+)$ s.t.\ $v_\lambda-u_\lambda\in{\rm int}(\cs_+)$, in particular $u_\lambda<v_\lambda$ in $\Omega$. Also, Lemma \ref{exi} \ref{exi3} says that for all $0<\lambda<\mu<\lambda^*$ we have $u_\mu-u_\lambda\in{\rm int}(\cs_+)$, in particular $u_\lambda<u_\mu$ in $\Omega$. This proves \ref{bif1}.
\vskip2pt
\noindent
By Lemma \ref{trs}, as $\lambda\to\lambda^*$ we have $u_\lambda\to u^*$, with $u^*\in{\rm int}(\cs_+)$ solution of $(P_{\lambda^*})$. This proves \ref{bif2}.
\vskip2pt
\noindent
Finally, by Lemma \ref{exi} \ref{exi1} we have $\lambda^*<\infty$, and by \eqref{str} for all $\lambda>\lambda^*$ there is no positive solution to $(P_\lambda)$. This proves \ref{bif3}.
\end{proof}

\begin{example}\label{ccr}
We collect here some functions $f:\Omega\times\R^+\times\R^+_0\to\R$ satisfying hypotheses ${\bf H}$ (as usual we assume $f(x,t,\lambda)=0$ for all $x\in\Omega$, $t<0$, and $\lambda>0$:
\begin{itemize}
\item[$(a)$] (non-autonomous concave-convex reaction) let $1<q<p<r<p^*_s$, $a,b\in L^\infty(\Omega)$ s.t.\ $a\ge a_0$, $b\ge b_0$ in $\Omega$ for some $a_0,b_0>0$, and set for all $(x,t,\lambda)\in\Omega\times\R^+\times\R^+_0$
\[f(x,t,\lambda) = \lambda a(x)t^{q-1}+b(x)t^{r-1};\]
\item[$(b)$] (autonomous reaction) let $1<q<p<r<(Np+p^2s)/N$, and set for all $(t,\lambda)\in\R^+\times\R^+_0$
\[f(t,\lambda) = \begin{cases}
\lambda t^{q-1} & \text{if $t\in[0,1]$} \\
\displaystyle \lambda t^{p-1}(\ln(t)+1) & \text{if $t>1$,}
\end{cases}\]
noting that $f$ does not satisfy the classical Ambrosetti-Rabinowitz condition.
\end{itemize}
\end{example}

\noindent
Notably, our approach also works in the case when $f(x,\cdot,\lambda)$ is asymptotically $(p-1)$-linear at the origin:

\begin{remark}\label{lin}
Assume that ${\bf H}$ holds, just replacing hypothesis ${\bf H}$ \ref{h4} with the following:
\begin{enumerate}[label=$(\roman*$)]
\setcounter{enumi}{3}
\item \label{h4lin} for all $\Lambda>0$ there exist $\hat\lambda_1<c_2\le c_3$ s.t.\ uniformly for a.e.\ $x\in\Omega$ and all $\lambda\in(0,\Lambda]$
\[c_2 \le \liminf_{t\to 0^+}\frac{f(x,t,\lambda)}{t^{p-1}} \le \limsup_{t\to 0^+}\frac{f(x,t,\lambda)}{t^{p-1}} \le c_3.\]
\end{enumerate}
Then, all the conclusions of Theorem \ref{bif} hold. Indeed, there are only two main steps at which the arguments for the present case differ from those seen above. The first is in the proof of Lemma \ref{ssp}, in proving that
\[\inf_{\w}\bar\Phi_\lambda < 0.\]
Indeed, fix $\eps>0$ s.t.\
\[\eps < c_2-\hat\lambda_1.\]
We can find $\delta>0$ s.t.\ for a.e.\ $x\in\Omega$ and all $t\in[0,\delta]$
\[f(x,t,\lambda) > (c_2-\eps)t^{p-1}.\]
Further, since $e_1\in{\rm int}(\cs_+)$, find $\tau>0$ s.t.\ in $\Omega$
\[0 < \tau e_1 \le \min\{\delta,\bar{u}\}.\]
By de l'H\^opital's rule, we have
\begin{align*}
\bar\Phi_\lambda(\tau e_1) &\le \frac{\tau^p\|e_1\|^p}{p}-\int_\Omega(c_2-\eps)\frac{(\tau e_1)^p}{p}\,dx \\
&\le \frac{\tau^p}{p}\big(\hat\lambda_1-c_2+\eps\big) < 0.
\end{align*}
A second difference appears in the proof of Lemma \ref{exi}, precisely in proving that $\lambda^*<\infty$. Using \ref{h4lin} in the place of ${\bf H}$ \ref{h4} we easily obtain \eqref{exi6}, and the rest follows as above.
\vskip2pt
\noindent
An example of (autonomous) reaction satisfying the modified hypotheses is the following: let $1<q<p<r<p^*$, $\eta>\hat\lambda_1$, and set for all $(t,\lambda)\in \R^+\times\R^+_0$
\[f(x,t,\lambda) = \begin{cases}
\lambda t^{r-1}+\eta t^{p-1} & \text{if $t\in[0,1]$} \\
\displaystyle \lambda t^{q-1}+\eta t^{r-1} & \text{if $t>1$.}
\end{cases}\]
\end{remark}

\vskip4pt
\noindent
{\bf Acknowledgement.} Both authors are members of GNAMPA (Gruppo Nazionale per l'Analisi Matematica, la Probabilit\`a e le loro Applicazioni) of INdAM (Istituto Nazionale di Alta Matematica 'Francesco Severi') and are supported by the research projects {\em Evolutive and Stationary Partial Differential Equations with a Focus on Biomathematics} and {\em Analysis of Partial Differential Equations in Connection with Real Phenomena}, funded by Fondazione di Sardegna (2019, 2021).

\end{document}